\documentclass[letterpaper,11pt,twoside,keywordsasfootnote,addressatend,noinfoline]{article}
\usepackage{fullpage}
\usepackage[english]{babel}
\usepackage{amssymb}
\usepackage{amsmath}
\usepackage{amsthm}
\usepackage{epsfig}
\usepackage{subfigure}
\usepackage{mathrsfs}
\usepackage{color}
\usepackage{imsart}
\usepackage{enumerate}
\usepackage{comment}

\newtheorem{theorem}{Theorem}
\newtheorem{lemma}[theorem]{Lemma}
\newtheorem{definition}[theorem]{Definition}
\newtheorem{proposition}[theorem]{Proposition}

\newcommand{\N}{\mathbb{N}}
\newcommand{\Z}{\mathbb{Z}}
\newcommand{\R}{\mathbb{R}}

\newcommand{\Lat}{\mathscr{L}}

\newcommand{\norm}[1]{|\!|#1|\!|}
\newcommand{\ind}{\mathbf{1}}
\newcommand{\ep}{\epsilon}
\newcommand{\n}{\hspace*{-5pt}}
\newcommand{\blue}[1]{\textcolor{blue}{#1}}

\usepackage{mathtools}

\DeclareMathOperator{\card}{card}

\DeclareMathOperator{\uniform}{Uniform}

\begin{document}

\begin{frontmatter}
\title{Evolutionary games on the lattice: multitype contact \\ process with density-dependent birth rates}
\runtitle{Multitype contact process with  density-dependent birth rates}
\author{Jonas K\"oppl, Nicolas Lanchier and Max Mercer}
\runauthor{Jonas K\"oppl, Nicolas Lanchier, and Max Mercer}
\address{Weierstrass Institute \\ 10117 Berlin, Germany. \\ koeppl@wias-berlin.de}
\address{School of Mathematical and Statistical Sciences \\ Arizona State University \\ Tempe, AZ 85287, USA. \\ nicolas.lanchier@asu.edu \\ mamerce1@asu.edu}
\maketitle

\begin{abstract} \
 Interacting particle systems of interest in evolutionary game theory introduced in the probability literature consist of variants of the voter model in which each site is occupied by one player.
 The goal of this paper is to initiate the study of evolutionary games based more realistically on the multitype contact process in which each site is either empty or occupied by a player following one of two possible competing strategies.
 Like in the symmetric multitype contact process, players have natural death rate one and natural birth rate~$\lambda$.
 Following the traditional modeling approach of evolutionary game theory, the process also depends on a payoff matrix~$A = (a_{ij})$ where~$a_{ij}$ represents the payoff a type~$i$ player receives from each of its type~$j$ neighbors, and the actual birth rate is an increasing function of the payoff.
 Using various couplings and block constructions, we first prove the existence of a phase transition in the direction of the intra payoff~$a_{11}$ or~$a_{22}$ while the other three payoffs are fixed.
 We also look at the behavior near the critical point where all four payoffs are equal to zero, in which case the system reduces to the symmetric multitype contact process.
 The effects of the intra payoffs~$a_{11}$ and~$a_{22}$ are studied using various couplings and duality techniques, while the effects of the inter payoffs~$a_{12}$ and~$a_{21}$ are studied in one dimension using a coupling with the contact process to control the interface between the~1s and the~2s.
\end{abstract}



\end{frontmatter}


\section{Introduction}
\label{sec:intro}
 The field of evolutionary game theory was developed by Maynard Smith~\cite{maynardsmith_1982}, and first appeared in his work with Price~\cite{maynardsmith_price_1973}.
 The basic idea in this field is to reinterpret the different strategies as species and their payoff as fitness  to create realistic dynamical systems with density-dependent birth and/or death rates.
 The earliest model in the field of evolutionary game theory is the~(nonspatial deterministic) replicator equation.
 Having an~$n \times n$ payoff matrix~$A = (a_{ij})$ where~$a_{ij}$ represents the payoff a type~$i$ player receives from a type~$j$ player, and letting~$u_i$ denote the density of type~$i$ players in the population, the payoff of each type~$i$ player is given by
 $$ \phi_i = \phi_i (u_1, u_2, \ldots, u_n) = a_{i1} u_1 + a_{i2} u_2 + \cdots + a_{in} u_n. $$
 Reinterpreting the payoff as fitness~(a birth rate when the payoff is positive or minus a death rate when the payoff is negative) and assuming that each individual produced replaces a player chosen uniformly at random, while each individual removed is replaced by a player chosen uniformly at random, result in the following so-called replicator equation~\cite{hofbauer_sigmund_1998}:
 $$ \begin{array}{c} u_i' = (\phi_i u_i)(\sum_{j \neq i} u_j) - (\sum_{j \neq i} \phi_j u_j) \,u_i = \sum_{j \neq i} (\phi_i - \phi_j) \,u_i u_j \quad \hbox{for} \quad i = 1, 2, \ldots, n. \end{array} $$
 This system of coupled differential equations can be turned into a spatially explicit stochastic process following the modeling approach of~\cite{nowak_2006, nowak_may_1992}.
 More precisely, to include a spatial structure in the form of local interactions, we first assume that the players are located on the~$d$-dimensional integer lattice~$\Z^d$, making the state at time~$t$ a spatial configuration
 $$ \xi_t : \Z^d \longrightarrow \{1, 2, \ldots, n \} \quad \hbox{where} \quad \xi_t (x) = \hbox{strategy of the player at site~$x$}. $$
 Then, writing~$x \sim y$ to indicate that the two lattice points~$x$ and~$y$ are nearest neighbors~(distance one apart), the payoff of the player at site~$x$ is defined as
\begin{equation}
\label{eq:payoff}
\begin{array}{c} \phi (x, \xi_t) = \sum_{i, j} a_{ij} \,f_j (x, \xi_t) \,\ind \{\xi_t (x) = i \} \quad \hbox{where} \quad f_j (x, \xi_t) = \sum_{y \sim x} \ind \{\xi_t (y) = j \} / 2d \end{array}
\end{equation}
 denotes the fraction of nearest neighbors of site~$x$ following strategy~$j$, i.e., the payoff only depends on the strategy of the neighbors.
 The fitness of the player at~$x$ is then given by
 $$ \Phi (x, \xi_t) = (1 - w) \times 1 + w \times \phi (x, \xi_t) = (1 - w) \times 1 + w \times \hbox{payoff}, $$
 where the parameter~$w \in [0, 1]$ represents the strength of selection.
 Weak selection refers to the case where~$w$ is small, while strong selection means that~$w = 1$.
 The most popular models that fall in this framework are the birth-death updating process and the death-birth updating process introduced in~\cite{ohtsuki_al_2006}.
 In the birth-death updating process, the fitness is interpreted as a birth rate, and offspring replace a neighbor of the parent's site chosen uniformly at random, so the rate at which site~$x$ switches from strategy~$i$ to strategy~$j$ is given by
\begin{equation}
\label{eq:birth-death}
\begin{array}{c} c_{i \to j} (x, \xi_t) = \sum_{y \sim x} \Phi (y, \xi_t) \,\ind \{\xi_t (y) = j \} / 2d \quad \hbox{for all} \quad i \neq j. \end{array}
\end{equation}
 In contrast, in the death-birth updating process, players die at rate one and are instantaneously replaced by the offspring of a neighbor chosen at random with a probability proportional to its fitness, so the local transition rates are given by
\begin{equation}
\label{eq:death-birth}
\begin{array}{c} c_{i \to j} (x, \xi_t) = \sum_{y \sim x} \Phi (y, \xi_t) \,\ind \{\xi_t (y) = j \} / \sum_{y \sim x} \Phi (y, \xi_t) \quad \hbox{for all} \quad i \neq j. \end{array}
\end{equation}
 Taking~$w = 0$, the transition rates in~\eqref{eq:birth-death}--\eqref{eq:death-birth} simplify to~$f_j (x, \xi_t)$, showing that both processes reduce to the voter model~\cite{clifford_sudbury_1973, holley_liggett_1975}.
 In the presence of weak selection, these two processes were studied in~\cite{chen_2013, chen_2018, ma_durrett_2018} in the context of two-strategy games, while~\cite{cox_durrett_2016, durrett_2014, nanda_durrett_2017} also considered games with more strategies such as rock-paper-scissors.
 In the weak selection limit~$w \to 0$, voter model perturbations techniques developed in~\cite{cox_durrett_perkins_2013} can be used to have a precise description of the phase structure of the processes.
 In the presence of strong selection~$w = 1$, more qualitative aspects such as the existence of phase transitions were proved in~\cite{evilsizor_lanchier_2016, lanchier_2015}.
 Other natural variants of these models with discontinuous transition rates were also studied rigorously in the presence of strong selection: the best-response dynamics~\cite{evilsizor_lanchier_2014}, and the death-birth of the fittest process~\cite{foxall_lanchier_2017}.
 For more details about these models, we refer the reader to~\cite[Chapter~7]{lanchier_2024}.


\section{Model description and main results}
\label{sec:results}
 In the previous models, each birth/death induces the instantaneous death/birth of a neighbor to ensure that each site is occupied by exactly one player at all times.
 These models and all the models of interacting particle systems of interest in evolutionary game theory that have been studied in the probability literature, consist of variants of the voter model with density-dependent birth and/or death rates.
 As far as we know, the only exception is the variant of Neuhauser's multitype contact process~\cite{neuhauser_1992} introduced in~\cite{lanchier_2019}.
 This process, however, was only designed to model the interactions among cooperators and defectors in the prisoner's dilemma rather than general games described by a payoff matrix.
 The main objective of this paper is to initiate the study of spatial evolutionary games based more realistically on the multitype contact process instead of the voter model.
 In particular, focusing for simplicity on two-strategy games, the state at time~$t$ is now
 $$ \xi_t : \Z^d \longrightarrow \{0, 1, 2 \} \quad \hbox{where} \quad \xi_t (x) = \hbox{strategy of the player at site~$x$}, $$
 with the convention~0~=~empty.
 The dynamics combines the dynamics of the multitype contact process and the dynamics of the birth-death updating process~\eqref{eq:birth-death}.
 Like in the symmetric multitype contact process, we assume that, regardless of their strategy, the players have the same natural birth rate~$\lambda$ and the same natural death rate one, and that births onto already occupied sites are suppressed.
 Having a payoff matrix~$A = (a_{ij})$, each player now receives a payoff from its occupied neighbors while empty neighbors have no effects, so the payoff can be defined as in~\eqref{eq:payoff} assuming that~$a_{i0} = 0$, i.e., empty sites give a zero payoff.
 Like in the birth-death updating process, the payoff of the players affects their birth rate.
 To have a well-defined positive birth rate even when the payoff is negative, we assume that the natural birth rate~$\lambda$ of the players is multiplied by the exponential of their payoff.
 In particular, the transition rates are given by
\begin{equation}
\label{eq:multitype}
\begin{array}{c} c_{0 \to i} (x, \xi_t) = \sum_{y \sim x} \Phi (y, \xi_t) \,\ind \{\xi_t (y) = i \} / 2d \quad \hbox{and} \quad c_{i \to 0} (x, \xi_t) = 1 \end{array}
\end{equation}
 for all~$i \neq 0$, where the fitness function is defined as
 $$ \begin{array}{c} \Phi (x, \xi_t) = \lambda \exp (\phi (x, \xi_t)) = \lambda \exp \big(\sum_{i, j \neq 0} a_{ij} \,f_j (x, \xi_t) \,\ind \{\xi_t (x) = i \} \big). \end{array} $$
 Note that, because~$\exp (0)$ = 1, players with no neighbors give birth at rate~$\lambda$, while players with a positive/negative payoff give birth at a higher/lower rate.
 Even though we choose the function~$\exp$ to fix the ideas, our results hold more generally for any function increasing from zero to~$+ \infty$, and equal to one at zero.
\begin{figure}[t]
\centering
\scalebox{0.90}{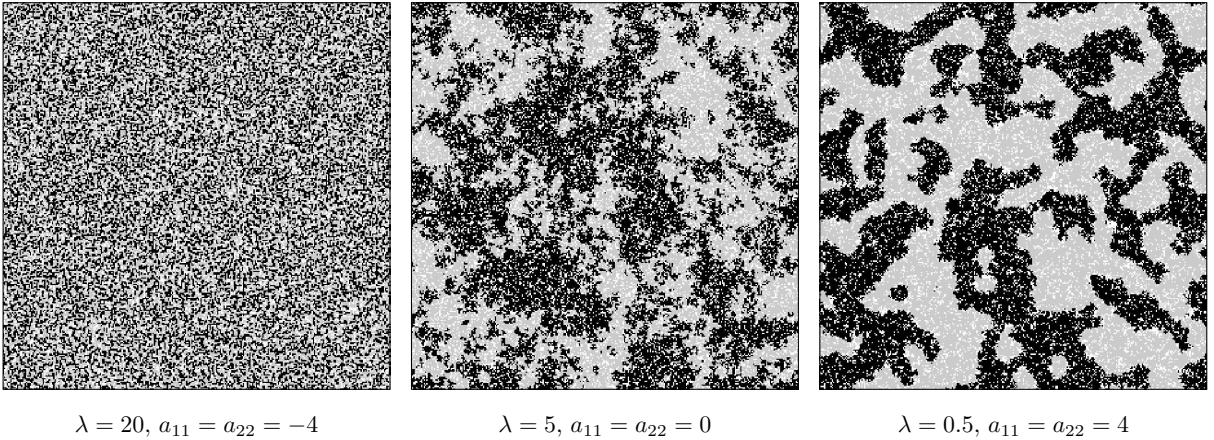}
\caption{\upshape
 Snapshots at time~1000 of the process~\eqref{eq:multitype} on the~$300 \times 300$ torus with inter payoffs~$a_{12} = a_{21} = 0$ and various values of the intra payoffs and the birth rate.
 The type~1 players are represented in black, the type~2 players in grey, and the empty sites in white.
 When~$a_{11} = a_{22} = 0$~(middle panel), the process reduces to the symmetric multitype contact process in which clustering occurs.
 In contrast, when~$a_{11} = a_{22} < 0$~(left panel), players help the opposite type, which results in coexistence, but when~$a_{11} = a_{22} > 0$~(right panel), players help their own type, which results in a clustering more pronounced than in the symmetric case.}
\label{fig:MCP}
\end{figure}
 Figure~\ref{fig:MCP} shows simulations of the two-dimensional process for various values of the natural birth rate and the intra payoffs while the inter payoffs are set to zero. \\
\indent
 To characterize the limiting behavior of the process, we say that a strategy starting from a positive density survives, respectively, dies out, if its density stays bounded away from zero, respectively, converges to zero, as time goes to infinity.
 In the presence of both strategies, we say that a strategy wins if it survives while the other strategy dies out, and that both strategies coexist if they both survive.
 From now on, we let~$\lambda_c$ denote the critical value of Harris' contact process~\cite{harris_1974} on the~$d$-dimensional integer lattice.
 In preparation of this paper, we studied in~\cite{koppl_lanchier_mercer_2024} the conditions for survival/extinction of the process~\eqref{eq:multitype} in the presence of only one strategy, say strategy~1, in which case the model only depends on two parameters:
 the birth rate~$\lambda$ and the payoff~$a_{11} = a$.
 We proved that, for all~$a \in \R$, there is at least one phase transition from extinction to survival in the direction of the birth rate~$\lambda$, while, for all~$\lambda > 0$, there is at least one phase transition from extinction to survival in the direction of the payoff~$a$.
 The phase transition with respect to the payoff~$a$ is proved using various block constructions to compare the process with oriented site percolation, which implies that, for all~$\lambda > 0$, there exist~$- \infty < a_- \leq a_+ < + \infty$ such that
\begin{equation}
\label{eq:survival-extinction}
\begin{array}{rcl}
\hbox{the process decays exponentially} & \hbox{for all} & a < a_-, \vspace*{4pt} \\
\hbox{the process grows linearly} & \hbox{for all} & a > a_+. \end{array}
\end{equation}
 The proof of the exponential decay relies more specifically on a perturbation argument and the fact that, when~$a_{11} = - \infty$, the~1s  behave like a system of coalescing random walks with killings, which remains true in the presence of~2s whenever~$a_{12} < \infty$.
 Using also the linear growth of the~2s in the absence of~1s when~$\lambda > \lambda_c$ and~$a_{22} \geq 0$, and the fact that, with positive probability, a linearly growing cluster of~2s never interacts with the exponentially decaying~1s, we deduce that
\begin{theorem}
\label{th:decay}
 Assume that~$\lambda > \lambda_c$ and~$a_{22} \geq 0$. Then,
 $$ \hbox{there exists~$a_- > - \infty$ such that the 2s win for all~$a_{11} < a_-$.} $$
\end{theorem}
\noindent
 More generally, the~2s win when they are supercritical~(in the absence of~1s), and the~1s decay exponentially, which happens when~$a_{11} < a_-$.
 Looking now at the process with~$a_{11}$ large, one expects the~1s to win even when the~2s are supercritical.
 The block construction showing linear growth~\eqref{eq:survival-extinction} of the~1s in the absence of the~2s does not imply the previous result because the~2s now prevents the~1s from expanding freely, but the proof extends to obtain that
\begin{theorem}
\label{th:a11large}
 Assume that~$\lambda > 0$. Then,
 $$ \hbox{there exists~$a_+ < + \infty$ such that the~1s win for all~$a_{11} > a_+$.} $$
\end{theorem}
\noindent
 The idea is to show that, with probability close to one, a~$2 \times 2 \times \cdots \times 2$ block of~1s survives at least until the surrounding~2s die, at which time the block expands.
 This and a block construction show that the~1s not only grow linearly but also exclude the~2s.
\begin{figure}[t!]
\centering
\scalebox{0.40}{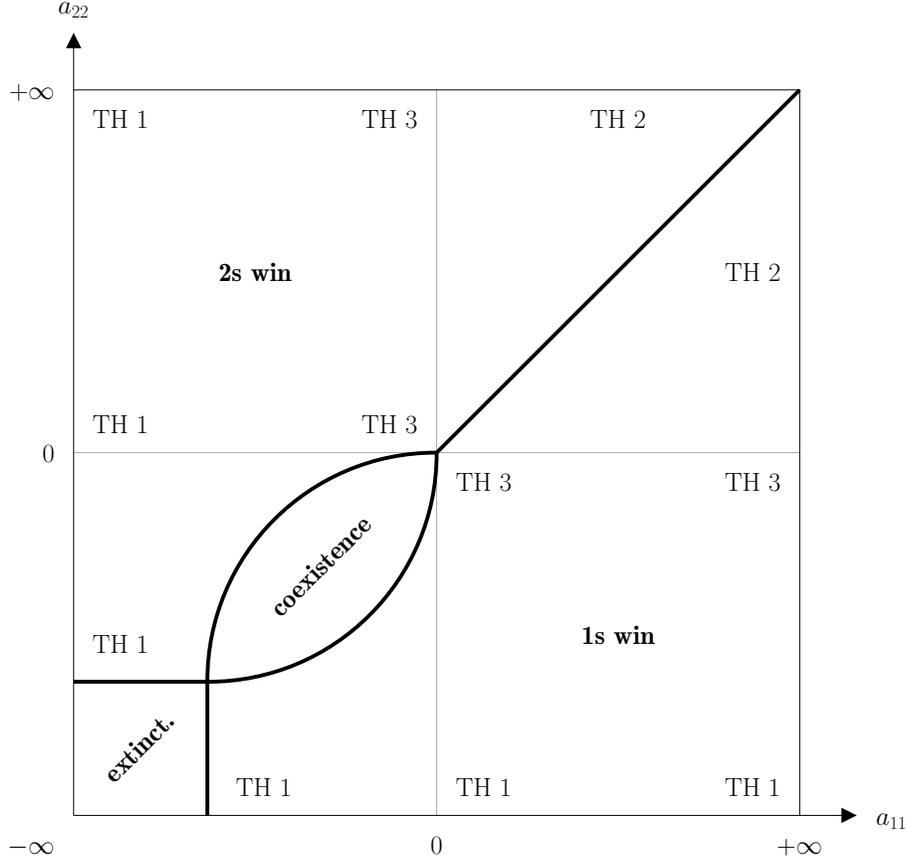}
\caption{\upshape{
 Phase structure of the multitype contact process with density-dependent birth rates in the~$a_{11}-a_{22}$ plane when~$\lambda > \lambda_c$ and~$a_{12} = a_{21} = 0$ suggested by numerical simulations of the two-dimensional process.
 The figure also underlines the parameter regions covered in Theorems~\ref{th:decay}, \ref{th:a11large}, and~\ref{th:a11>0}.}}
\label{fig:phase1}
\end{figure}
\begin{figure}[t!]
\centering 
\scalebox{0.40}{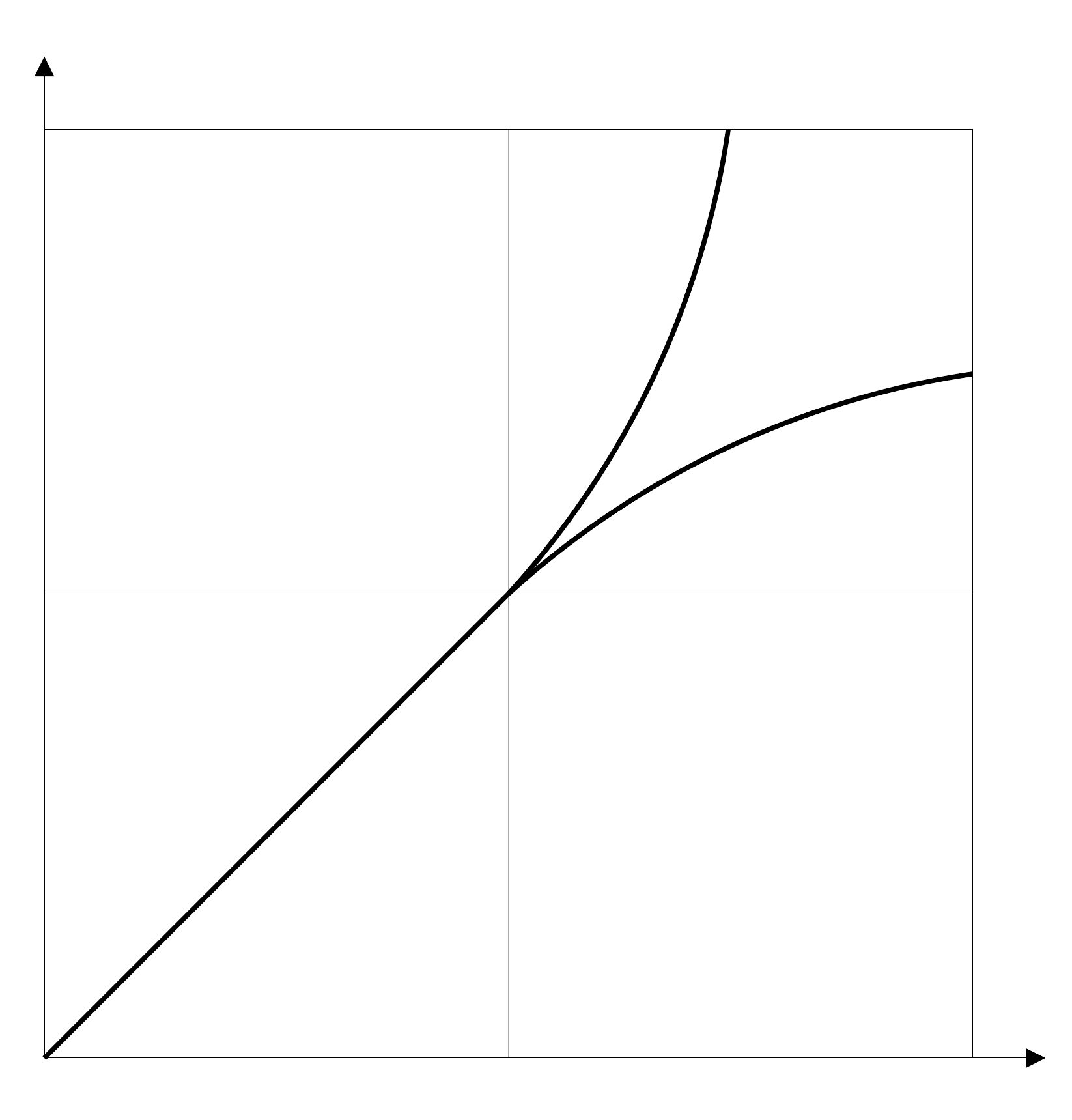}
\caption{\upshape{
 Phase structure of the multitype contact process with density-dependent birth rates in the~$a_{12}-a_{21}$ plane when~$\lambda > \lambda_c$ and~$a_{11} = a_{22} = 0$ suggested by numerical simulations of the two-dimensional process.
 We conjecture that there is no coexistence region in the one-dimensional nearest-neighbor case.}}
\label{fig:phase2}
\end{figure}
 We now look at parameter regions that extend near the critical point where all the payoffs are equal to zero, in which case the process reduces to the symmetric multitype contact process.
 To begin with, we study the effects of the intra payoff~$a_{11}$ that models the influence of the~1s on the nearby~1s.
 Assuming that~$\lambda > \lambda_c$ and~$a_{11} > 0$ while the other three payoffs are equal to zero, the set of~1s dominates stochastically its counterpart~(and the set of~2s is dominated stochastically by its counterpart) in a variant of the multitype contact process where pairs of adjacent~1s give birth at a rate larger than~$\lambda$.
 This process has a tractable dual process that can be used to prove that the~1s win.
 Using a coupling argument to compare processes with different payoff matrices shows that the result still holds if we increase~$a_{12}$ and decrease~$a_{21}$ and~$a_{22}$, from which it follows that
\begin{theorem}
\label{th:a11>0}
 Assume that~$\lambda > \lambda_c$, $a_{11} > 0$, and~$a_{12} \geq 0 \geq a_{21}, a_{22}$. Then, the~1s win.
\end{theorem}
\noindent
 We conjecture that strategy~2 wins when~$a_{11} < 0$ and~$a_{12} \leq 0 \leq a_{21}, a_{22}$ but the couplings and duality techniques used to prove the previous theorem fail to establish this result due in part to a lack of attractiveness:
 when~$a_{11} < 0$, increasing the local density of the~1s decreases the birth rate of the surrounding~1s, which decreases their local density.
 To conclude, we study the effects of the inter payoff~$a_{12}$ that models the influence of the~2s on the nearby~1s.
 Note that, from the point of view of the inter payoff, the lack of attractiveness happens when~$a_{12} > 0$ since in this case increasing the local density of the~2s increases the birth rate of the surrounding~1s, which decreases the local density of the~2s.
 The process is difficult to study more generally when~$a_{12} \neq 0$ while the other three payoff are equal to zero due to the lack of a mathematically tractable dual process and the fact that the interface between both strategies~(the set of edges that connect a~1 and a~2) is difficult to control/quantify.
 However, the analysis is simplified when~$d = 1$ and we start with only~1s on the negative integers and only~2s on the positive integers since in this case the interface is either empty or reduces to a single edge.
 More precisely, we can prove that
\begin{theorem}
\label{th:a12}
 Let~$\lambda > \lambda_c$ and~$\xi_0 (x) = \ind \{x \leq 0 \} + 2 \times \ind \{x > 0 \}$. Then,
 $$ \hbox{the~1s win when~$a_{12} > a_{21} \geq 0$ and~$a_{11} \geq a_{22} \geq 0$}. $$
\end{theorem}
\noindent
 The proof relies on a coupling with the contact process to control the interface between the~1s and the~2s.
 When the~1s and the~2s are not in contact, they evolve according to two independent supercritical contact processes so the next time they come in contact is finite on average.
 In addition, because~$a_{11} \geq a_{22}$, the configuration of~1s is stochastically larger than or equal to the mirror image of the configuration of~2s.
 When the~1s and the~2s are in contact, because~$a_{12} > a_{21}$, the rightmost~1 has a higher birth rate than but the same death rate as the leftmost~2 so at the time they separate~(when one of them dies), the configuration of~1s is stochastically strictly larger than the mirror image of the configuration~2s.
 This shows the existence of a positive drift implying that the interface converges to infinity, which proves the theorem.
 For pictures of the phase structure of the process suggested by numerical simulations and a summary of our results, see Figures~\ref{fig:phase1} where the inter payoffs are equal to zero and Figure~\ref{fig:phase2} where the intra payoffs are equal to zero.
 The rest of the paper is devoted to the construction of the process and the proof of Theorems~\ref{th:a11large}--\ref{th:a12}.


\section{Graphical representation}
\label{sec:coupling}
 This section shows how to construct the process graphically from a collection of independent Poisson processes following an idea of Harris~\cite{harris_1978}.
 To generate the potential births, we let
 $$ M = \lambda \exp (0 \vee a_{11} \vee a_{12} \vee a_{21} \vee a_{22}), $$
 which represents the largest possible birth rate of a player over all possible configurations.
 Then, for each oriented edge~$\vec{xy}$, we draw an arrow~$x \to y$ at the times~$B_n (\vec{xy})$ of a rate~$M / 2d$ Poisson process to indicate a potential birth from site~$x$ to site~$y$.
 More precisely, to each
 $$ (x, B_n (\vec{xy})) \to (y, B_n (\vec{xy})), \quad \hbox{we attach} \quad U_n (\vec{xy}) = \uniform (0, M), $$
 and, letting~$t = B_n (\vec{xy})$, call this arrow~$i$-open if
 $$ U_n (\vec{xy}) \leq \lambda \exp (a_{i1} f_1 (x, \xi_{t-}) + a_{i2} f_2 (x, \xi_{t -})). $$
 Finally, we assume that, if there is a type~$i$ player at the tail~$x$ and the head~$y$ is empty at time~$t-$, then site~$y$ becomes occupied by a type~$i$ player if and only if the arrow is~$i$-open.
 To generate the potential deaths, for each~$x \in \Z^d$, we put a cross at~$x$ at the times~$D_n (x)$ of a rate one Poisson process to indicate that a player at site~$x$ is killed at time~$t = D_n (x)$.
 Using this graphical representation, we can compare processes with different parameters.
 More precisely, let
 $$ \begin{array}{rcl}
    \xi_t & \n = \n & \hbox{process with parameters} \ \lambda, a_{11}, a_{12}, a_{21}, a_{22}, \vspace*{4pt} \\
    \bar \xi_t & \n = \n & \hbox{process with parameters} \ \lambda, \bar a_{11}, \bar a_{12}, \bar a_{21}, \bar a_{22}, \end{array} $$
 and for~$i = 1, 2$, denote by
 $$ \xi_t^i = \{x \in \Z^d : \xi_t (x) = i \} \quad \hbox{and} \quad \bar \xi_t^i = \{x \in \Z^d : \bar \xi_t (x) = i \} $$
 the set of sites occupied by the type~$i$ players at time~$t$ in each of the two processes.
\begin{proposition}
\label{prop:coupling}
 Assume that
 $$ \bar a_{11} \geq a_{11} \geq 0, \quad \bar a_{12} \geq 0 \geq a_{12}, \quad \bar a_{21} \leq 0 \leq a_{21}, \quad \bar a_{22} \leq 0 \leq a_{22}. $$
 Then, there exists a coupling~$(\xi_t, \bar \xi_t)$ such that
 $$ \bar \xi_0^1 \supset \xi_0^1 \quad \hbox{and} \quad \bar \xi_0^2 \subset \xi_0^2 \quad \Longrightarrow \quad \bar \xi_t^1 \supset \xi_t^1 \quad \hbox{and} \quad \bar \xi_t^2 \subset \xi_t^2 \quad \hbox{for all} \ t \geq 0. $$
\end{proposition}
\begin{proof}
 We can couple the two processes by constructing them simultaneously from the graphical representation described above with
 $$ \begin{array}{rcl}
      M & \n = \n & \lambda \exp (0 \vee a_{11} \vee a_{12} \vee a_{21} \vee a_{22} \vee \bar a_{11} \vee \bar a_{12} \vee \bar a_{21} \vee \bar a_{22}) \vspace*{4pt} \\
        & \n = \n & \lambda \exp (\bar a_{11} \vee \bar a_{12} \vee a_{21} \vee a_{22}). \end{array} $$
 Assume that the two inclusions are true until right before time~$t = B_n (\vec{xy})$. Then,
\begin{equation}
\label{eq:coupling-1}
\begin{array}{l}
\lambda \exp (\bar a_{11} f_1 (x, \bar \xi_{t-}) + \bar a_{12} f_2 (x, \bar \xi_{t -})) \geq \lambda \exp (\bar a_{11} f_1 (x, \bar \xi_{t-})) \vspace*{4pt} \\ \hspace*{50pt} \geq
\lambda \exp (a_{11} f_1 (x, \xi_{t-})) \geq \lambda \exp (a_{11} f_1 (x, \xi_{t-}) + a_{12} f_2 (x, \xi_{t-})). \end{array}
\end{equation}
 Similarly, for the type~2 players, we have
\begin{equation}
\label{eq:coupling-2}
\begin{array}{l}
\lambda \exp (\bar a_{21} f_1 (x, \bar \xi_{t-}) + \bar a_{22} f_2 (x, \bar \xi_{t -})) \vspace*{4pt} \\ \hspace*{50pt} \leq
\lambda \leq \lambda \exp (a_{21} f_1 (x, \xi_{t-}) + a_{22} f_2 (x, \xi_{t -})). \end{array}
\end{equation}
 The inclusion~$\bar \xi_{t-}^1 \supset \xi_{t-}^1$ and the inequalities in~\eqref{eq:coupling-1} show that, if
 $$ \xi_{t-} (x) = 1, \quad \xi_{t-} (y) = 0, \quad (x, t) \to (y, t) \ \hbox{is 1-open for the process} \ \xi, $$
 in which case~$y$ becomes occupied by a 1, the same holds for the process~$\bar \xi$ unless there is already a type~1 player at site~$y$, so the inclusion still holds at time~$t$. 
 Similarly, the inclusion~$\bar \xi_{t-}^2 \subset \xi_{t-}^2$ and the inequalities in~\eqref{eq:coupling-2} show that, if
 $$ \bar \xi_{t-} (x) = 2, \quad \bar \xi_{t-} (y) = 0, \quad (x, t) \to (y, t) \ \hbox{is 2-open for the process} \ \bar \xi, $$
 in which case~$y$ becomes occupied by a 2, the same holds for the process~$\xi$ unless there is already a type~2 player at site~$y$, so the inclusion still holds at time~$t$.
 Since crosses have the same effect on both processes, the two inclusions also remain true if~$t$ is the time of a cross.
\end{proof}


\section{Proof of Theorem~\ref{th:a11large}}
\label{sec:a11large}
 This section is devoted to the proof of Theorem~\ref{th:a11large}, which states that the birth rate~$\lambda > 0$ and the other three payoffs being fixed, strategy~1 wins whenever~$a_{11}$ is sufficiently large.
 The proof is based on a standard block construction.
 Define~$\Lambda_- = \{0, 1 \}^d$ and the cross-shaped region
 $$ \begin{array}{c} \Lambda_+ = \{x \in \Z^d : \min_{y \in \Lambda_-} \norm{x - y}_2 \leq 1 \}. \end{array} $$
 See Figure~\ref{fig:cross} for a picture in dimension two.
 The first step to prove the theorem is to show that, starting with only type~1 players in~$\Lambda_-$, for~$a_{11}$ large, these~1s will expand to occupy~$\Lambda_+$ and exclude the~2s with probability close to one.
 Regardless of their birth rate, the~1s can only spread if the~2s around die to give them space, and they can only persist if they are not subject to too many deaths, so the first step is to control the number of crosses/death marks in the cross-shaped region~$\Lambda_+$.
 More precisely, we have the following lemma where
 $$ N_x (0, T) = \max \{n : D_n (x) < T \} = \hbox{number of crosses at~$x$ by time~$T$}. $$
\begin{lemma}
\label{lem:death}
 For all~$\ep > 0$, there exists~$T = T (\ep)$ and~$N = N (\ep, T)$ such that
 $$ P (0 < N_x (0, T) < N \ \hbox{for all} \ x \in \Lambda_+) \geq 1 - \ep / 4. $$
\end{lemma}
\begin{proof}
 To begin with, we let~$\ep > 0$ and define
 $$ a = \frac{\ep}{8 \card (\Lambda_+)} = \frac{\ep}{8 (2^d + 2d \times 2^{d - 1})} = \frac{\ep}{(d + 1) \,2^{d + 3}} > 0. $$
 Because the random variable~$N_x (0, T)$ is Poisson distributed with mean~$T$, we can choose the parameter~$T$ sufficiently large that
\begin{equation}
\label{eq:death-1}
  P (N_x (0, T) = 0) = \exp (-T) \leq a.
\end{equation}
 Now, with time~$T < \infty$ being fixed, there exists~$N < \infty$ sufficiently large such that
\begin{equation}
\label{eq:death-2}
\begin{array}{c}  P (N_x (0, T) \geq N) = \exp (-T) \sum_{k \geq N} T^k / k! \leq a. \end{array}
\end{equation}
 Combining~\eqref{eq:death-1} and~\eqref{eq:death-2}, we deduce that
 $$ \begin{array}{l}
      P (N_x (0, T) \not \in (0, N) \ \hbox{for some} \ x \in \Lambda_+) \leq
    \sum_{x \in \Lambda_+} P (N_x (0, T) \not \in (0, N)\blue{)} \vspace*{4pt} \\ \hspace*{25pt} \leq
    \sum_{x \in \Lambda_+} P (N_x (0, T) = 0) + \sum_{x \in \Lambda_+} P (N_x (0, T) \geq N)) \leq 2a \card (\Lambda_+) = \ep / 4, \end{array} $$
 which proves the lemma.
\end{proof}
\noindent
 Now that we have some control over the number of death marks, we can prove that, starting with only type~1 players in~$\Lambda_-$, with probability close to one, these players will expand to~$\Lambda_+$ in less than~$T$ units of time and exclude the~2s from~$\Lambda_+$ for another~$T$ units of time.
\begin{figure}[t!]
\label{fig:cross}
\centering
\scalebox{0.50}{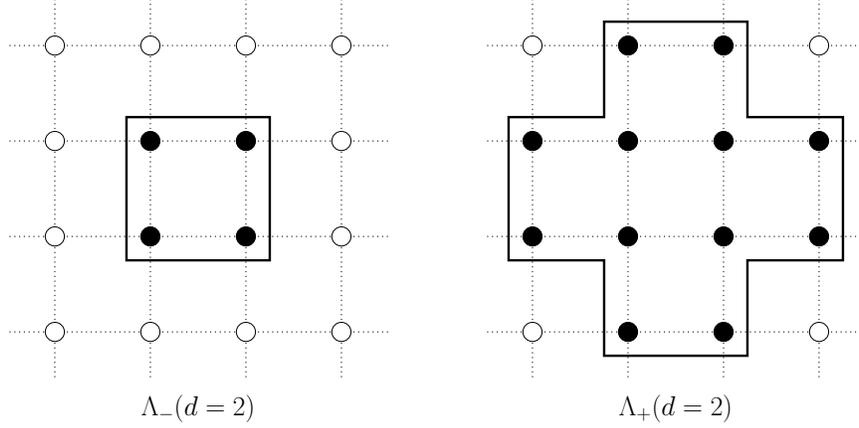}
\caption{\upshape{
 Picture of the region~$\Lambda_-$ on the left and picture of the region~$\Lambda_+$ in the middle in two dimensions, and illustration of the events in Lemma~\ref{lem:invade} on the right in one dimension.}}
\end{figure}
\begin{lemma}
\label{lem:invade}
 For all~$\ep, \lambda > 0$ and~$- \infty < a_{12}, a_{21}, a_{22} < \infty$,
 $$ P (\Lambda_+ \subset \xi_T^1 \ \hbox{and} \ \Lambda_+ \cap \xi_t^2 = \varnothing \ \hbox{for all} \ t \in [T, 2T) \,| \,\Lambda_- \subset \xi_0^1) \geq 1 - \ep $$
 for all~$a_{11} < \infty$ sufficiently large.
\end{lemma}
\begin{proof}
 To begin with, note that, when~$a_{11} \geq a_{12} \vee 0$ and site~$x$ has at least one type~1 neighbor, the rate of~1-open arrows~$x \to y$ is larger than
\begin{equation}
\label{eq:invade-1}
  M_1 = M_1 (a_{11}, a_{12}) = \frac{\lambda}{2d} \,\exp \bigg(\frac{a_{11}}{2d} + (2d - 1) \,\frac{a_{12} \wedge 0}{2d} \bigg),
\end{equation}
 while the rate of~2-open arrows pointing at~$y$ is smaller than
\begin{equation}
\label{eq:invade-2}
  M_2 = M_2 (a_{21}, a_{22}) = \lambda \exp (0 \vee a_{21} \vee a_{22}).
\end{equation}
 To shorten the notation, we also let
 $$ B_{N, T} = \{0 < N_x (0, T) < N \ \hbox{for all} \ x \in \Lambda_+ \}. $$
 Now, observe that each site in~$\Lambda_+$ has at least one neighbor in~$\Lambda_-$ and that, if~$N_x (0, T) \neq 0$, site~$x$ is empty at least once before time~$T$.
 In particular, starting with~$\Lambda_-$ fully occupied by~1s, to prove that~$\Lambda_+$ is fully occupied by~1s at time~$T$, it suffices to prove that, between any two consecutive death marks in~$\Lambda_+$, there is a~1-open arrow from~$\Lambda_-$ to the location, say~$y$, of the first death mark before any~2-open arrows from~$\Lambda_+$ to site~$y$.
 Thus, the lower/upper bounds~\eqref{eq:invade-1} and~\eqref{eq:invade-2}, and the fact that death marks in~$\Lambda_+$ appear at rate~$\card (\Lambda_+)$, imply that
\begin{equation}
\label{eq:invade-3}
\begin{array}{rcl}
  P (\Lambda_+ \not \subset \xi_T^1 \,| \,B_{N, T}) & \n \leq \n &
  N \card (\Lambda_+) (M_2 + \card (\Lambda_+)) / (M_1 + M_2 + \card (\Lambda_+)) \vspace*{4pt} \\ \hspace*{40pt} & \n \leq \n &
  N \card (\Lambda_+) (M_2 + \card (\Lambda_+)) / M_1 \leq \ep / 4 \end{array}
\end{equation}
 for all~$a_{11}$ sufficiently large since~$M_1 = M_1 (a_{11}, a_{12}) \to \infty$ as~$a_{11} \to \infty$, while~$N, M_2 (a_{21}, a_{22})$ and the spatial dimension are fixed.
 Using~\eqref{eq:invade-3} and Lemma~\ref{lem:death}, we deduce that
\begin{equation}
\label{eq:invade-4}
  P (\Lambda_+ \subset \xi_T^1) \geq P (\Lambda_+ \subset \xi_T^1 \,| \,B_{N, T}) P (B_{N, T}) \geq (1 - \ep / 4)^2 \geq 1 - \ep / 2.
\end{equation}
 Repeating the exact same reasoning also implies that
\begin{equation}
\label{eq:invade-5}
  P (\Lambda_+ \cap \xi_t^2 = \varnothing \ \hbox{for all} \ t \in (T, 2T) \,| \,\Lambda_+ \subset \xi_T^1) \geq 1 - \ep / 2.
\end{equation}
 Finally, combining~\eqref{eq:invade-4} and~\eqref{eq:invade-5}, we conclude that, starting with~$\Lambda_- \subset \xi_0^1$,
 $$ \begin{array}{l}
      P (\Lambda_+ \subset \xi_T^1 \ \hbox{and} \ \Lambda_+ \cap \xi_t^2 = \varnothing \ \hbox{for all} \ t \in (T, 2T)) \vspace*{4pt} \\ \hspace*{40pt} =
      P (\Lambda_+ \cap \xi_t^2 = \varnothing \ \hbox{for all} \ t \in (T, 2T) \,| \,\Lambda_+ \subset \xi_T^1) P (\Lambda_+ \subset \xi_T^1) \vspace*{4pt} \\ \hspace*{40pt} \geq
        (1 - \ep / 2)^2 \geq 1 - \ep. \end{array} $$
 This proves the lemma.
\end{proof}
\noindent
 To complete the proof of the theorem, the last step is to compare the process properly rescaled in space and time with oriented site percolation. Let
 $$ \Lat = \{(m, n) \in \Z^d \times \N : m_1 + \cdots + m_d + n \ \hbox{is even} \}, $$
 and turn~$\Lat$ into a directed graph~$\vec{\Lat}$ by placing an edge
 $$ (m, n) \to (m', n') \quad \hbox{if and only if} \quad \norm{m_i' - m_i} = 1 \quad \hbox{and} \quad n' = n + 1. $$
 Fix~$\ep > 0$, let~$T = T (\ep)$ as in Lemma~\ref{lem:death}, and call~$(m, n) \in \Lat$ a \textit{good site} if
\begin{itemize}
\item $E_{m, n} = \{m + \Lambda_- \subset \Lambda_{nT}^1  \}$ and \vspace*{4pt}
\item $F_{m, n} = \{(m + \Lambda_-) \cap \xi_t^2 = \varnothing \ \hbox{for all} \ t \in [nT, (n + 1) T) \}$
\end{itemize}
 both occur.
 Note that, accounting also for translation invariance of the graphical representation in space and time, the conclusion of Lemma~\ref{lem:invade} can be written as
 $$ P (E_{m', n'} \cap F_{m', n'} \ \hbox{for all} \ (m', n') \leftarrow (m, n) \,| \,E_{m, n}) \geq 1 - \ep. $$
 In particular, it follows from~\cite[Theorem~A.4]{durrett_1995} that the set of good sites dominates stochastically the set of wet sites in an oriented site percolation process on~$\vec{\Lat}$ with parameter~$1 - \ep$ and finite range of dependence.
 In view of the events~$E_{m, n}$,
 choosing~$\ep > 0$ small to ensure percolation of the open sites implies survival of the type~1 players.
 However, because there is a positive density~$\ep > 0$ of closed sites, and the corresponding space-time blocks may contain type~2 players, this does not show extinction of the~2s.
 To also prove that the~2s die out, we use a percolation result in~\cite{durrett_1992} which shows that, when~$\ep$ is small enough, not only the open sites percolate but also the closed sites do not percolate.
 In view of the events~$F_{m, n}$, and the fact that
 $$ \bigcup_{(m, n) \in \Lat} ((m + \Lambda_-) \times [nT, (n + 1) T)) = \Z^d \times \R_+, $$
 the presence of a type~2 player after time~$nT$ in the interacting particle system implies the presence of a directed path of closed sites of length at least~$n$ in the percolation process.
 This shows extinction of the~2s and completes the proof of Theorem~\ref{th:a11large}.


\section{Proof of Theorem~\ref{th:a11>0}}
\label{sec:a11>0}
 This section is devoted to the proof of Theorem~\ref{th:a11>0}, which states that when~$\lambda$ is supercritical for the basic contact process and~$a_{12} \geq 0 \geq a_{21}, a_{22}$, strategy~1 wins whenever~$a_{11} > 0$.
 The proof is based on coupling arguments and duality-like techniques.
 Note that, in view of Proposition~\ref{prop:coupling}, it suffices to prove that the~1s win for the process~$\xi_t$ with
 $$ \lambda > \lambda_c, \quad a_{11} > 0, \quad \hbox{and} \quad a_{12} = a_{21} = a_{22} = 0, $$
 in which case type~2 players give birth at rate~$\lambda$ regardless of their surrounding environment.
 To further simplify the proof of the theorem, we will first deal instead with the process~$\zeta_t$ constructed from the following graphical representation:
\begin{itemize}
\item
 For each site~$x$, put a death mark~$\times$ at site~$x$ at the times of a rate one Poisson process to indicate the death of a player of either type at~$x$. \vspace*{4pt}
\item
 For each site~$x$ and each neighbor~$y$, draw a single arrow~$x \to y$ at the times of a rate~$\lambda / 2d$ Poisson process to indicate that if site~$x$ is occupied by a player of either type and site~$y$ is empty then the player at~$x$ gives birth through the arrow. \vspace*{4pt}
\item
 For each site~$x$ and each ordered pair~$(y, z)$ of neighbors of~$x$, draw a double arrow~$z \to x \to y$ at the times of a Poisson process with rate
 $$ \ep = \frac{\lambda}{(2d)^2} \,\bigg(\exp \bigg(\frac{a_{11}}{2d} \bigg) - 1 \bigg) > 0 $$
 to indicate that if sites~$x$ and~$z$ are both occupied by a type~1 player and site~$y$ is empty then the player at site~$x$ gives birth onto site~$y$.
\end{itemize}
 The reason for introducing this process is that it can be studied using duality-like techniques.
 In addition, the parameters~$\lambda > \lambda_c$ and~$a_{11} > 0$ being fixed, if the~1s win for~$\zeta_t$ then the~1s also win for the process~$\xi_t$, which is a direct consequence of the following proposition.
\begin{proposition}
\label{prop:more-coupling}
 There exists a coupling~$(\xi_t, \zeta_t)$ such that
\begin{equation}
\label{eq:xi-zeta}
\zeta_0^1 \subset \xi_0^1 \quad \hbox{and} \quad \zeta_0^2 \supset \xi_0^2 \quad \Longrightarrow \quad \zeta_t^1 \subset \xi_t^1 \quad \hbox{and} \quad \zeta_t^2 \supset \xi_t^2 \quad \hbox{for all} \ t \geq 0.
\end{equation}
\end{proposition}
\begin{proof}
 Assume that~$\xi_t$ is constructed from the graphical representation described at the beginning of Section~\ref{sec:coupling} with~$M = \lambda \exp (a_{11})$.
 In particular, an arrow~$x \to y$ at time~$t = B_n (\vec{xy})$ is
 $$ \begin{array}{rcl}
    \hbox{1-open for~$\xi_t$} & \Longleftrightarrow & U_n (\vec{xy}) \leq \lambda \exp (a_{11} f_1 (x, \xi_{t-})), \vspace*{4pt} \\
    \hbox{2-open for~$\xi_t$} & \Longleftrightarrow & U_n (\vec{xy}) \leq \lambda. \end{array} $$
 To also construct the process~$\zeta_t$, to each pair
 $$ (B_n (\vec{xy}), U_n (\vec{xy})), \quad \hbox{we attach} \quad V_n (\vec{xy}) = \uniform \{x - e_1, x + e_1, \ldots, x- e_d, x + e_d \}. $$
 Then, we draw an additional arrow~$(V (\vec{xy}), B_n (\vec{xy})) \to (x, B_n (\vec{xy}))$ if and only if
 $$ \lambda < U_n (\vec{xy}) \leq \lambda \exp (a_{11} / 2d), $$
 and assume that~$(x, B_n (\vec{xy})) \to (y, B_n (\vec{xy}))$ is~1-open for~$\zeta_t$ if and only if
 $$ U_n (\vec{xy}) \leq \lambda \quad \hbox{or} \quad (\lambda < U_n (\vec{xy}) \leq \lambda \exp (a_{11} / 2d) \ \ \hbox{and} \ \ \zeta_{t-} (V_n (\vec{xy})) = 1). $$
 The~2-open arrows for the process~$\zeta_t$ are defined as for the process~$\xi_t$.
 In addition, for both processes, players of either type are killed by the crosses, and type~$i$ players give birth through the~$i$-open arrows.
 Note that the rate of a double arrow~$z \to x \to y$ is given by
 $$ \begin{array}{l}
    \displaystyle \frac{M}{2d} \ P \bigg(\lambda < U_n (\vec{xy}) \leq \lambda \exp \bigg(\frac{a_{11}}{2d} \bigg) \ \hbox{and} \ V_n (\vec{xy}) = z \bigg) \vspace*{8pt} \\ \hspace*{30pt} =
    \displaystyle \frac{M}{2d} \ \bigg(\frac{\lambda}{M} \,\bigg(\exp \bigg(\frac{a_{11}}{2d} \bigg) - 1 \bigg) \bigg) \ \frac{1}{2d} =
    \displaystyle \frac{\lambda}{(2d)^2} \,\bigg(\exp \bigg(\frac{a_{11}}{2d} \bigg) - 1 \bigg) = \ep, \end{array} $$
 so this construction indeed generates the process~$\zeta_t$.
 Since the~1-open arrows corresponding to a single arrow for~$\zeta_t$, the~2-open arrows, and the crosses have the same effect on both processes, in order to show the two inclusions, it suffices to focus on the effect of the double arrows.
 Assume that the inclusions are true until right before time~$t = B_n (\vec{xy})$, and that at this time there is a double arrow~$z \to x \to y$ and that a~1 at site~$x$ gives birth onto site~$y$ for the process~$\zeta_t$, i.e.,
 $$ \zeta_{t-} (z) = \zeta_{t-} (x) = 1, \quad \zeta_{t-} (y) = 0, \quad \lambda < U_n (\vec{xy}) \leq \lambda \exp (a_{11} / 2d). $$
 Using that~$\xi_{t-}$ has more~1s than~$\zeta_{t-}$ implies that~$\xi_{t-} (z) = 1$ hence~$f_1 (x, \xi_{t-}) \geq 1/2d$.
 Using also that~$\xi_{t-}$ has less~2s than~$\zeta_{t-}$, we deduce that
 $$ \xi_{t-} (z) = \xi_{t-} (x) = 1, \quad \xi_{t-} (y) \neq 2, \quad U_n (\vec{xy}) \leq \lambda \exp (a_{11} / 2d) \leq \lambda \exp (a_{11} f_1 (x, \xi_{t-})), $$
 the last inequality showing that the arrow~$x \to y$ is also~1-open for the process~$\xi_t$.
 In particular, either site~$y$ is already occupied by a~1 at time~$t-$ or the~1 at site~$x$ gives birth onto the empty site~$y$ at time~$t$.
 In both cases, the inclusions are preserved.
\end{proof}
\noindent
 Even though the process~$\zeta_t$ gives an advantage to the~1s, because isolated~1s give birth at the same rate as the~2s, how to prove that the~1s win is unclear.
 To show this result, notice that, each time the graphical representation at~$x, y, z$, where~$y$ and~$z$ are distinct neighbors of site~$x$, consists of a death mark at~$z$ then a single arrow~$x \to z$ then a double arrow~$z \to x \to y$, and no other death marks or arrows pointing at any of these three sites in this time window, a player of type~1 at site~$x$ will give birth onto~$z$ through the single arrow and then onto~$y$ through the second piece~$x \to y$ of the double arrow.
 However, the~2s cannot give birth through this arrow.
 See Figure~\ref{fig:arrows} for an illustration.
\begin{figure}[t!]
\label{fig:arrows}
\centering
\scalebox{0.50}{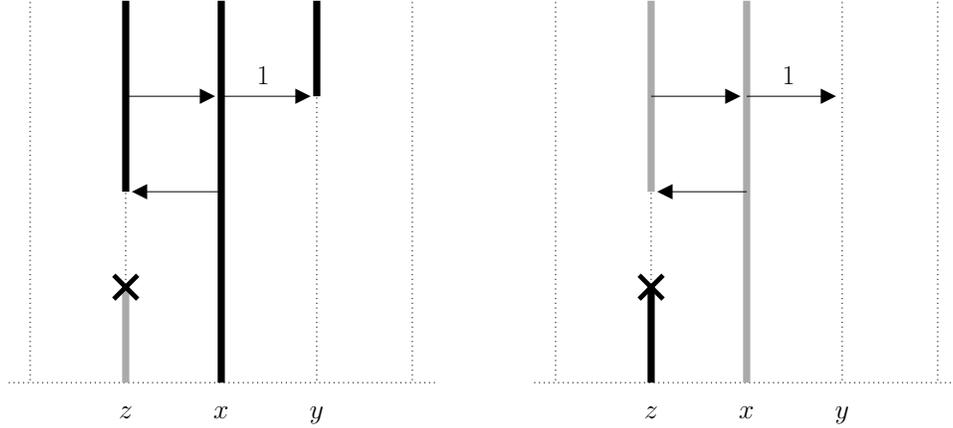}
\caption{\upshape{
 Illustration of the effect of a~1-arrow.
 In both pictures, black refers to a type~1 player, while gray refers to a type~2 player.
 A type~1 player at~$x$ can give birth through the arrow~$x \to y$ if there is a type~1 player at~$z$ but, regardless of the surrounding configuration, type~2 players cannot give birth through the arrow~$x \to y$.}}
\end{figure}
 This motivates the following definition.
\begin{definition}
\label{def:good}
 Assuming that there is a double arrow~$z \to x \to y$ with~$y \neq z$ at some time~$s > 2$, we say that the double arrow is good and label~$x \to y$ with a~1 whenever
\begin{itemize}
 \item there is a death mark at site~$z$ between time~$s - 2$ and time~$s - 1$, \vspace*{4pt}
 \item there is a single arrow~$x \to z$ between time~$s - 1$ and time~$s$, \vspace*{4pt}
 \item there are no other death marks at, single arrows pointing at, or double arrows pointing at sites~$x, y$ or~$z$ between time~$s - 2$ and time~$s$.
\end{itemize}
\end{definition}
\noindent
 As previously explained, the~1s can give birth through the~1-arrows, i.e, the second piece of a good double arrow, but not the~2s.
 In addition, there is a positive density of~1-arrows in the sense that we have the following lemma.
\begin{lemma}
\label{lem:1-arrows}
 Assume that, for~$i = 1, 2, \ldots, n$, there is a double arrow~$z_i \to x_i \to y_i$ with~$y_i \neq z_i$ at some time~$s_i > 2$.
 Then, there exists~$c > 0$ such that
 $$ P ((z_i, s_i) \to (x_i, s_i) \to (y_i, s_i) \ \hbox{is a good for~$i = 1, 2, \ldots, n$}) \geq c^n > 0 $$
 whenever~$\norm{x_i - x_j}_1 \vee |s_i - s_j| > 2$ for all~$i \neq j$.
\end{lemma}
\begin{proof}
 Note that whether a double arrow is good or not depends on the graphical representation but not on the configuration.
 Also, letting~$N_{\mu}$ be the Poisson random variable with mean~$\mu$, observing that single/double arrows pointing at a given site occur at rate~$\lambda$ and~$(2d)^2 \ep$, respectively, and using independence of the Poisson processes in the graphical representation, we deduce that the probability that the double arrow~$z \to x \to y$ at time~$s$ is a good arrow is
 $$ \begin{array}{rcl}
      c & \n = \n & P (N_1 = 1) \,P (N_{\lambda} = 1) \,P (N_5 = 0) \,P (N_{5 \lambda} = 0) \,P (N_{6 (2d)^2 \ep} = 0) \vspace*{4pt} \\
        & \n = \n & \lambda \exp (- 6 (1 + \lambda + (2d)^2 \ep)) > 0. \end{array} $$
 In addition, under the assumption of the lemma about the space-time distance between the arrows, whether the double arrows~$z_i \to x_i \to y_i$ at time~$s_i$ are good or not are determined by disjoint parts of the graphical and are therefore independent, from which it follows that
 $$ P ((z_i, s_i) \to (x_i, s_i) \to (y_i, s_i) \ \hbox{is good for~$i = 1, 2, \ldots, n$}) = c^n > 0. $$
 This completes the proof.
\end{proof}
\noindent
 The~1s can give birth through the second piece of a double arrow if~(but not only if) the double arrow is good.
 To prove that the~1s win, we consider a third process~$\eta_t$ coupled with~$\zeta_t$, obtained by using the~1-arrows only.
 More precisely, having a graphical representation of the process~$\zeta_t$ with~1-arrows as in Definition~\ref{def:good}, the process~$\eta_t$ is constructed as follows:
\begin{itemize}
\item
 Players of either type at site~$x$ are killed by a death mark~$\times$ at that site. \vspace*{4pt}
\item
 Players of either type at site~$x$ give birth through a single arrow~$x \to y$ if site~$y$ is empty. \vspace*{4pt}
\item
 Players of type~1 at site~$x$ give birth through a~1-arrow~$x \to y$ if site~$y$ is empty.
\end{itemize}
 Because death marks, single arrows, and~1-arrows have the same effect on both processes, while the~1s in the process~$\zeta_t$ can give birth through additional arrows~(second piece of double arrows that are not necessarily good), the process~$\eta_t$ has less~1s and more~2s, i.e.,
\begin{equation}
\label{eq:zeta-eta}
\eta_0^1 \subset \zeta_0^1 \quad \hbox{and} \quad \eta_0^2 \supset \zeta_0^2 \quad \Longrightarrow \quad \eta_t^1 \subset \zeta_t^1 \quad \hbox{and} \quad \eta_t^2 \supset \zeta_t^2 \quad \hbox{for all} \ t \geq 0.
\end{equation}
 In particular, to complete the proof of the theorem, the last step is to show that the~1s win in the process~$\eta_t$, which relies on duality.
 To describe the dual process of~$\eta_t$, we say that, for~$s < t$, there is a path~$(y, s) \uparrow (x, t)$ if there are times and sites
 $$ s = s_0 < s_1 < \cdots < s_{n + 1} = t \quad \hbox{and} \quad y = x_0, x_1, ..., x_n = x $$
 such that the following two conditions hold:
\begin{itemize}
\item for~$i = 1, 2, \ldots, n$, there is a single arrow or a~1-arrow~$x_{i - 1} \to x_i$ at time~$s_i$ and \vspace*{4pt}
\item for~$i = 0, 1, \ldots, n$, the time segments~$\{x_i \} \times (s_i, s_{i + 1}$ do not contain any death mark~$\times$.
\end{itemize}
 Then, we say that there is a dual path~$(x, t) \downarrow (y, t - s)$ if there is a path~$(y, t - s) \uparrow (x, t)$ and, for each~$(x, t) \in \Z^d \times \R_+$, define the dual process
 $$ \hat{\eta}_s (x, t) = \{y \in \Z^d : (x, t) \downarrow (y, t - s) \} \quad \hbox{for all} \quad 0 \leq s \leq t. $$
 This process is similar to the dual process of the contact process.
 In particular, like for the contact process,the space-time region 
 $$ \Gamma = \{(\hat \eta_s (x, t), s): 0 \leq s \leq t \}, $$
 exhibits a tree structure, and the players in the set~$\hat \eta_t (x, t)$ at time zero represent the potential ancestors of the player at~$(x, t)$ if this space-time location is indeed occupied.
 However, in contrast with the basic contact process, the ancestors can now be of two types, and the type of~$(x, t)$ depends on the type and location of the ancestors.
 More precisely, Neuhauser~\cite{neuhauser_1992} proved that the sites in the dual process at time~0~(in fact the dual paths going from~$(x, t)$ to time~0) can be arranged in the order they determine the type of~$(x, t)$, which she referred to as the ancestor hierarchy and which depends on the topology of the tree structure~$\Gamma$.
 We refer to~\cite{neuhauser_1992, durrett_neuhauser_1997} for a definition of the ancestor hierarchy and to~\cite[Section~1.4.1]{lanchier_2024} for another more visual description.
 Then, the type of the player at site~$x$ at time~$t$~(when~$(x, t)$ is indeed occupied) can be determined as follows:
\begin{itemize}
\item
 If the first dual path in the hierarchy lands on an empty site, then we look at the next dual path in the hierarchy until the first dual path that lands on an occupied site. \vspace*{4pt}
\item
 If the first dual path landing on an occupied site lands on a~1, then~$\eta_t (x) = 1$. \vspace*{4pt}
\item
 If this dual path lands on a~2, then~$\eta_t (x) = 2$ unless the dual path crosses a~1-arrow, in which case we look at the last~1-arrow the dual path crosses and remove all the dual paths going through the tail of this arrow, which are the next dual paths in the hierarchy.
\end{itemize}
 Because there is a positive density of~1-arrows, it follows from the proof of~\cite[Theorem~1]{neuhauser_1992} that the~1s win in the process~$\eta_t$.
 This, together with Proposition~\ref{prop:coupling} and the couplings~\eqref{eq:xi-zeta} and~\eqref{eq:zeta-eta}, implies that the~1s also win in the process~$\xi_t$ under the assumptions of Theorem~\ref{th:a11>0}.


\section{Proof of Theorem~\ref{th:a12}}
 This section is devoted to the proof of Theorem~\ref{th:a12}, which states that, starting with only~1s on the left of and at the origin and only~2s on the right of the origin in~$d = 1$, the~1s win when
\begin{equation}
\label{eq:cond-a12}
\lambda > \lambda_c, \quad a_{11} \geq a_{22} \geq 0, \quad \hbox{and} \quad a_{12} > a_{21} \geq 0.
\end{equation}
 In this case, because the births of players with two occupied neighbors are unimportant and the type~1 players are always on the left of the type~2 players, the process can be constructed from the following ``simplified'' graphical representation:
\begin{itemize}
\item
 For each site~$x$, put a death mark~$\times$ at site~$x$ at the times of a rate one Poisson process to indicate the death of a player of either type at~$x$. \vspace*{4pt}
\item
 For each site~$x$ and each neighbor~$x \pm 1$, draw a single arrow~$x \to x \pm 1$ at the times of a rate~$\lambda / 2$ Poisson process to indicate that if site~$x$ is occupied by a player of either type and site~$x \pm 1$ is empty then the player at~$x$ gives birth through the arrow. \vspace*{4pt}
\item
 For each site~$x$, draw a double~11-arrow~$x \pm 1 \overset{1 \ }{\longrightarrow} x \overset{1 \ }{\longrightarrow} x \mp 1$ at the times of a Poisson process with rate
 $\ep_{11} = \lambda (\exp (a_{11} / 2) - 1) / 2$ to indicate that if~$x \pm 1$ and~$x$ are occupied by a~1 and~$x \mp 1$ is empty then the player at~$x$ gives birth onto~$x \mp 1$. \vspace*{4pt}
\item
 For each site~$x$, draw a double~22-arrow~$x \pm 1 \overset{2 \ }{\longrightarrow} x \overset{2 \ }{\longrightarrow} x \mp 1$ at the times of a Poisson process with rate
 $\ep_{22} = \lambda (\exp (a_{22} / 2) - 1) / 2$ to indicate that if~$x \pm 1$ and~$x$ are occupied by a~2 and~$x \mp 1$ is empty then the player at~$x$ gives birth onto~$x \mp 1$. \vspace*{4pt}
\item
 For each~$x$, draw a double~12-arrow~$x + 1 \overset{1 \ }{\longrightarrow} x \overset{2 \ }{\longrightarrow} x - 1$ at the times of a Poisson process with rate~$\ep_{12} = \lambda (\exp (a_{12} / 2) - 1) / 2$ to indicate that if site~$x + 1$ is occupied by a~2, site~$x$ occupied by a~1, and site~$x - 1$ empty then the player at~$x$ gives birth onto~$x - 1$. \vspace*{4pt}
\item
 For each~$x$, draw a double~21-arrow~$x - 1 \overset{2 \ }{\longrightarrow} x \overset{1 \ }{\longrightarrow} x + 1$ at the times of a Poisson process with rate~$\ep_{21} = \lambda (\exp (a_{21} / 2) - 1) / 2$ to indicate that if site~$x - 1$ is occupied by a~1, site~$x$ occupied by a~2, and site~$x + 1$ empty then the player at~$x$ gives birth onto~$x + 1$.
\end{itemize}
 Note that the assumptions in~\eqref{eq:cond-a12} imply that~$\ep_{11} \geq \ep_{22} \geq 0$ and~$\ep_{12} > \ep_{21} \geq 0$.
 To prove the theorem, we first partition time into intervals where the rightmost~1 and the leftmost~2 are in contact and intervals where they are not in contact.
 More precisely, letting
 $$ R_t = \sup \{x \in \Z : \xi_t (x) = 1 \} \quad \hbox{and} \quad L_t = \inf \{x \in \Z : \xi_t (x) = 2 \} $$
 be the location of the rightmost~1 and the location of the leftmost~2 at time $t$, we define two sequences of stopping times~$(\sigma_i)_{i \geq 0}$ and~$(\tau_i)_{i \geq 1}$ recursively by letting~$\sigma_0 = \tau_0 = 0$,
 $$ \sigma_i = \inf \{t > \tau_{i - 1} : L_t - R_t > 1 \} \quad \hbox{and} \quad \tau_i = \inf \{t > \sigma_i : L_t - R_t = 1 \} \quad \hbox{for all} \quad i \geq 1. $$
 To quantify the invasion of strategy~1, we also introduce the process~$X_t$ that keeps track of the location of the interface/midpoint between the~1s and the~2s defined as
 $$ X_t = (R_t + L_t) / 2 \quad \hbox{for all} \quad t \geq 0. $$
 To prove that the~1s win, meaning that~$X_t$ converges almost surely to infinity, the key is to lower bound the spatial displacements of the interface between times~$\tau_i$ and~$\tau_{i + 1}$ and upper bound the temporal displacements~$\tau_{i + 1} - \tau_i$, which is done in the next two lemmas.
\begin{lemma}
\label{lem:space}
 There exists~$\ep > 0$ such that~$E (X_{\tau_{i + 1}} - X_{\tau_i}) \geq \ep$ for all~$i \geq 1$.
\end{lemma}
\begin{proof}
 Let~$\mu_t^1$ and~$\mu_t^2$ be respectively the distribution of~1s at time~$t$ viewed from the right edge~$R_t$ and the distribution of~2s at time~$t$ viewed from the left edge~$L_t$.
 We say that~$\mu_t^1$ dominates the mirror image of~$\mu_t^2$, which we write~$\mu_t^1 \geq s (\mu_t^2)$, when, for all~$A \subset \N$ finite,
\begin{equation}
\label{eq:space-1}
  P (\xi_t (R_t - x) = 1 \ \hbox{for all} \ x \in A) \geq P (\xi_t (L_t + x) = 2 \ \hbox{for all} \ x \in A).
\end{equation}
 Similarly, we say that~$\mu_t^1$ strictly dominates the mirror image of~$\mu_t^2$ and write~$\mu_t^1 > s (\mu_t^2)$ when the inequality above is strict.
 Assume that~$\mu_{\tau_i}^1 \geq s (\mu_{\tau_i}^2)$.
 Between time~$\tau_i$ and time~$\sigma_{i + 1}$, the rightmost~1 and the leftmost~2 are in contact.
 Using attractiveness, $\ep_{11} \geq \ep_{22}$ and~$\ep_{12} > \ep_{21}$, and the fact that~$\mu_{\tau_i}^1 \geq s (\mu_{\tau_i}^2)$ by assumption, implies that the stochastic domination still holds until time~$\sigma_{i + 1}$.
 More generally, since there is a positive density of empty sites and
 $$ \hbox{rate of} \ \,L_t \overset{1 \ }{\longrightarrow} R_t \overset{2 \ }{\longrightarrow} R_t - 1 = \ep_{12} > \ep_{21} =
    \hbox{rate of} \ \,R_t \overset{2 \ }{\longrightarrow} L_t \overset{1 \ }{\longrightarrow} L_t + 1, $$
 we have~$\mu_t^1 > s (\mu_t^2)$ until time~$\sigma_{i + 1}$.
 In particular, letting
 $$ R^2_t = \sup \{x < R_t : \xi_t (x) = 1 \} \quad \hbox{and} \quad L^2_t = \inf \{x > L_t : \xi_t (x) = 2 \} $$
 be the location of the second rightmost~1 and second leftmost~2,
 $$ P (|R_{\sigma_{i + 1}-} - R^2_{\sigma_{i + 1}-}| \geq n) < P (|L_{\sigma_{i + 1}-} - L^2_{\sigma_{i + 1}-}| \geq n) \quad \hbox{for all} \quad n \geq 1. $$
 Using the previous strict domination, we deduce that
\begin{equation}
\label{eq:space-2}
\begin{array}{rcl}
  4 \ep_i & \n = \n & E (L^2_{\sigma_{i + 1}-} - L_{\sigma_{i + 1}-}) + E (R^2_{\sigma_{i + 1}-} - R_{\sigma_{i + 1}-}) \vspace*{4pt} \\
          & \n = \n & E \,|L_{\sigma_{i + 1}-} - L^2_{\sigma_{i + 1}-}| - E \,|R_{\sigma_{i + 1}-} - R^2_{\sigma_{i + 1}-}| > 0. \end{array}
\end{equation}
 Now, between time~$\sigma_{i + 1}$ and time~$\tau_{i + 1}$, the rightmost~1 and the leftmost~2 are not in contact, so the~1s and the~2s evolve according to two independent processes.
 Because~$\ep_{11}, \ep_{22} \geq 0$, these two processes are again attractive, and since~$\ep_{11} \geq \ep_{22}$ and~$\mu_{\sigma_{i + 1}}^1 > s (\mu_{\sigma_{i + 1}}^2)$, we have that
\begin{equation}
\label{eq:space-3}
\mu_t^1 > s (\mu_t^2) \quad \hbox{for all} \quad t \in [\sigma_{i + 1}, \tau_{i + 1}] \quad \hbox{and} \quad E (X_{\tau_{i + 1}}) \geq E (X_{\sigma_{i + 1}}).
\end{equation}
 Combining~\eqref{eq:space-2}--\eqref{eq:space-3}, and using that, at time~$\sigma_{i + 1}$, either the rightmost~1 dies or the leftmost~2 dies, each with probability one-half, we deduce that
\begin{equation}
\label{eq:space-4}
\begin{array}{rcl}
\displaystyle E (X_{\tau_{i + 1}} - X_{\tau_i}) & \n = \n &
\displaystyle E (X_{\tau_{i + 1}} - X_{\sigma_{i + 1}}) + E (X_{\sigma_{i + 1}} - X_{\tau_i}) \geq E (X_{\tau_{i + 1}} - X_{\sigma_{i + 1}}) \vspace*{12pt} \\ & \n = \n &
\displaystyle \frac{1}{2} \ E \bigg(\frac{R_{\tau_{i + 1}-} + L^2_{\tau_{i + 1}-}}{2} - \frac{R_{\tau_{i + 1}-} + L_{\tau_{i + 1}-}}{2} \bigg) \vspace*{8pt} \\ && \hspace*{20pt} + \
\displaystyle \frac{1}{2} \ E \bigg(\frac{R^2_{\tau_{i + 1}-} + L_{\tau_{i + 1}-}}{2} - \frac{R_{\tau_{i + 1}-} + L_{\tau_{i + 1}-}}{2} \bigg) \vspace*{8pt} \\ & \n = \n &
\displaystyle E \bigg(\frac{L^2_{\tau_{i + 1}-} - L_{\tau_{i + 1}-}}{4} \bigg) + E \bigg(\frac{R^2_{\tau_{i + 1}-} - R_{\tau_{i + 1}-}}{4} \bigg) = \ep_i > 0. \end{array}
\end{equation}
 Because~$\mu_t^1 \geq s (\mu_t^2)$ at time~$\tau_0 = 0$ since in this case the two probabilities in~\eqref{eq:space-1} are both equal to one, and because the stochastic domination is preserved across time according to the first inequality in~\eqref{eq:space-3}, a simple induction shows that the inequality~\eqref{eq:space-4} holds for all~$i \geq 1$.
 Finally, because the two measures~$\mu_t^1$ and~$\mu_t^2$ both converge to a limit due to attractiveness, the sequence~$(\ep_i)_{i \geq 0}$ also converges to a limit therefore~$\ep = \inf_i \ep_i > 0$.
 The lemma holds for this~$\ep$.
\end{proof}
\noindent
 To complete the proof of the theorem, the last step is to show that
\begin{lemma}
\label{lem:time}
 There exists~$c < \infty$ such that~$E (\tau_{i + 1} - \tau_i) \leq c$ for all~$i \geq 1$.
\end{lemma}
\begin{proof}
 Note that time~$\sigma_{i + 1}$ is the first time after~$\tau_i$ at which either the~1 at site~$R_{\tau_i}$ dies or the~2 at site~$L_{\tau_i}$ dies.
 Because players die independently at rate one, the time increment~$\sigma_{i + 1} - \tau_i$ is exponentially distributed with parameter two therefore
\begin{equation}
\label{eq:time-1}
  E (\sigma_{i + 1} - \tau_i) = 1/2 \quad \hbox{for all} \quad i \geq 1.
\end{equation}
 Controlling the time increment~$\tau_{i + 1} - \sigma_{i + 1}$ is more difficult.
 To this end, let~$\eta_t$ be the contact process with parameter~$\lambda$ starting from the all occupied configuration constructed from the graphical representation of the multitype contact process~$\xi_t$ described above by only using the death marks~$\times$ and the single arrows~$x \to x \pm 1$.
 For the resulting coupling~$(\xi_t, \eta_t)$,
\begin{equation}
\label{eq:time-2}
  (\eta_t (x) = 1 \ \Longrightarrow \ \xi_t (x) > 0) \quad \hbox{for all} \quad (x, t) \in \Z \times \R_+.
\end{equation}
 Now, let~$\mu_t$ be the distribution of~$\eta_t$.
 By attractiveness, $\mu_t$ decreases to the upper invariant measure~$\bar \mu$ as time goes to infinity.
 In addition, because~$\lambda > \lambda_c$, it follows from~\cite[Corollary~4.1]{liggett_steif_2006} that there exists a positive constant~$\rho$ such that the upper invariant measure dominates stochastically the product measure~$\nu_{\rho}$ with density~$\rho$.
 In summary,
\begin{equation}
\label{eq:time-3}
  (\mu_t \geq \bar \mu \geq \nu_{\rho} \ \ \hbox{for all} \ \ t \geq 0) \quad \hbox{for some} \quad \rho > 0,
\end{equation}
 which, together with the implication~\eqref{eq:time-2}, implies that, at time~$\sigma_{i + 1}$, the distance between the rightmost~1 and the leftmost~2 decays exponentially, i.e.,
 $$ P (L_{\sigma_{i + 1}} - R_{\sigma_{i + 1}} > n + 1) \leq (1 - \rho)^n \quad \hbox{for all} \quad n \geq 1. $$
 The supercritical contact process exhibits linear growth, which suggests that the time it takes for the rightmost~1 and the leftmost~2 to come back together also decays exponentially.
 However, the family of the~1 at time~$R_{\sigma_{i + 1}}$ and/or the family of the~2 at time~$L_{\sigma_{i + 1}}$ may die out quickly, in which case it is unclear that the time increment~$\tau_{i + 1} - \sigma_{i + 1}$ indeed scales like~$L_{\sigma_{i + 1}} - R_{\sigma_{i + 1}}$.
 To control this time increment, we use that the contact process with~$\lambda > \lambda_c$ starting from a single site survives with probability~$\bar \rho > 0$.
 Let~$\zeta_t^x$ be the coupled contact processes starting from
 $$ \zeta_0^x = \{x \} \quad \hbox{for all} \quad x \in \Z $$
 and constructed from the part of the graphical representation of~$\eta_t$ located after the stopping time~$\sigma_{i + 1}$.
 Then, for all~$x \in \Z$, we define the events
 $$ O_x = \{\eta_{\sigma_{i + 1}} (x) = 1 \} \quad \hbox{and} \quad S_x = \{\zeta_t^x \neq \varnothing \ \hbox{for all} \ t \geq 0 \}. $$
 Because the graphical representation is translation-invariant and the Poisson processes are independent, it follows from Birkhoff's ergodic theorem and~\eqref{eq:time-3} that, as~$n \to \infty$,
 $$ \frac{1}{n} \,\sum_{x = 0}^{n - 1} \,\ind \{O_{R_{\tau_i} - x} \cap S_{R_{\tau_i} - x} \}, \
    \frac{1}{n} \,\sum_{x = 0}^{n - 1} \,\ind \{O_{L_{\tau_i} + x} \cap S_{L_{\tau_i} + x} \} \,\to \,P (O_0 \cap S_0) \geq \rho \bar \rho > 0 $$
 with probability one.
 In particular, we have~$E (\bar L_{\sigma_{i + 1}} - \bar R_{\sigma_{i + 1}}) \leq 1 + 2 / \rho \bar \rho$, where
 $$ \bar R_{\sigma_{i + 1}} = \sup \{x \leq R_{\tau_i} : O_x \cap S_x \ \hbox{occurs} \} \quad \hbox{and} \quad
    \bar L_{\sigma_{i + 1}} = \inf \{x \geq L_{\tau_i} : O_x \cap S_x \ \hbox{occurs} \}. $$
 In addition, by the shape theorem~\cite[Theorem~5]{bezuidenhout_grimmett_1990},
 $$ \begin{array}{rcll}
      r_t & \n = \n & \sup \{x \in \Z : \zeta_t^{\bar R_{\sigma_{i + 1}}} (x) = 1 \} & \hbox{drifts linearly to the right}, \vspace*{4pt} \\
      l_t & \n = \n & \inf \{x \in \Z : \zeta_t^{\bar L_{\sigma_{i + 1}}} (x) = 1 \} & \hbox{drifts linearly to the left}, \end{array} $$
 from which it follows that there exists~$C = C (\lambda) < \infty$ such that
\begin{equation}
\label{eq:time-4}
  E (T) \leq C / \rho \bar \rho \quad \hbox{where} \quad T = \inf \{t : r_t > l_t \}.
\end{equation}
 Returning to the process~$\xi_t$, let~$\theta_{i + 1}$ be the first time after~$\sigma_{i + 1}$ the rightmost~1 or leftmost~2 tries to give birth onto the leftmost~2 or rightmost~1.
 Clearly, we have~$\theta_{i + 1} \geq \tau_{i + 1}$.
 In addition, in view of the implication~\eqref{eq:time-2}, the rightmost~1 is on the right of~$r_t$ and the leftmost~2 is on the left of~$l_t$ until at least~$\theta_{i + 1}$ therefore~$\theta_{i + 1} - \sigma_{i + 1} \leq T$.
 This,~\eqref{eq:time-1} and~\eqref{eq:time-4} imply that
 $$ \begin{array}{rcl}
      E (\tau_{i + 1} - \tau_i) & \n = \n & E (\tau_{i + 1} - \sigma_{i + 1}) + E (\sigma_{i + 1} - \tau_i) \vspace*{4pt} \\
                                & \n \leq \n & E (\theta_{i + 1} - \sigma_{i + 1}) + 1/2 \leq E (T) + 1/2 \leq C / \rho \bar \rho + 1/2 < \infty. \end{array} $$
 Since the upper bound is uniform in~$i$, the lemma follows.
\end{proof}
\noindent
 Combining the previous two lemmas, we deduce that the interface~$X_t$ converges almost surely to infinity, which shows that the~1s win and completes the proof of Theorem~\ref{th:a12}.


\end{document}